\documentclass[10pt]{article}

\usepackage[title]{appendix}
\usepackage[margin=1in]{geometry}
\usepackage{enumerate, amsmath, amssymb,latexsym}
\usepackage{amsthm, color, graphicx, hyperref, amscd}
\usepackage{todonotes, comment}
\usepackage{bbm}
\newcommand\nc{\newcommand}
\nc\DMO{\DeclareMathOperator}
\nc\ignore[1]{}     % allows you to comment out part of the text 

\usepackage[normalem]{ulem}

\newcommand{\cP}{\mathcal{P}}

\newcommand{\cL}{\mathcal{L}}

\newcommand{\HH}{\mathcal{H}}

\newcommand{\F}{\mathbb{F}}
\newcommand{\E}{\mathbb{E}}

\newcommand{\PG}{\mathrm{PG}}

\DeclareMathOperator{\pr}{pr}

\newtheorem{thm}{thm}[section]

\newtheorem{conj}[thm]{Conjecture}

\theoremstyle{definition}

\DMO{\Tr}{Tr}        \DMO{\N}{\N}    \DMO{\ind}{ind}
\DMO{\AG}{{AG}}                        \DMO{\GL}{{GL}}
\DMO{\PGL}{{PGL}}                     \DMO{\PGaL}{P\Gamma L}
\DMO{\Gal}{Gal}                          \DMO{\GaL}{\Gamma L}

\nc{\uE}{{\textup E}}                              \nc{\uL}{{\textup L}}
\nc{\uS}{{\textup S}}                              \nc{\Zp}{{\mathbb Z}_p}
\nc{\cR}{{\mathcal R}}                            \nc{\cA}{{\mathcal A}}
\nc{\NN}{{\mathbb N}}
\nc{\cN}{{\mathcal N}}                           \nc{\cM}{{\mathcal M}}
\nc{\cG}{{\mathcal G}}                           \nc{\ts}{{\theta_s}}
\nc{\cE}{{\mathcal E}}                            \nc{\cF}{{\mathcal F}}

\nc{\wa}{{\widehat{\alpha}}}                 \nc{\wg}{{\widehat{\gamma}}}
\nc{\cJ}{{\mathcal J}}                            \nc{\cQ}{{\mathcal Q}}
\nc{\Fqt}{\F_{q^t}}                               \nc{\K}{{\mathbb K}}
\nc{\GF}{\hbox{{\rm GF}}}                    \nc{\st}{{\cal S}_t}
\nc{\so}{{\cal S}}                                  \nc{\sk}{{\cal S}_2}
\nc{\ad}{E}                                            \nc{\lmb}{\lambda}
\nc{\la}{\langle}                                     \nc{\ra}{\rangle}

\nc{\ZG}{\mathcal{Z}(\mathrm{\Gamma L})}
\nc{\G}{\mathrm{\Gamma L}}                \nc{\ep}{\epsilon}
\nc{\vv}{\mathbf v}  
\nc{\1}{\mathbbm 1}  
\nc{\hs}{{\hat{\sigma}}}

\newtheorem{theorem}[thm]{Theorem}
\newtheorem{problem}[thm]{Problem}
\newtheorem{lemma}[thm]{Lemma}
\newtheorem{corollary}[thm]{Corollary}
\newtheorem{construction}[thm]{Construction}

\newtheorem{proposition}[thm]{Proposition}
\newtheorem{result}[thm]{Result}

\newtheorem{defi}[thm]{Definition}

\newtheorem{remark}[thm]{Remark}
\usepackage{xcolor}
\newcommand{\zs}[1]{\textcolor{magenta}{\textbf{#1}}}
\newcommand{\zsst}[1]{%
	\bgroup
	\markoverwith{\textcolor{magenta}{\rule[0.5ex]{2pt}{1.2pt}}}%
	\ULon{#1}%
	\egroup
}

\title{Balanced intersection size distributions in finite geometries and vector spaces}
\author{Zolt\'an L\'or\'ant Nagy\thanks{
		E\"otv\"os Lor\'and University, Budapest, Hungary. The author is supported by the J\'anos Bolyai Scholarship of the Hungarian Academy of Sciences and by the NRDI EXCELLENCE-24 grant no. 151504 Combinatorics and Geometry. %Hungarian Research Grant (NKFIH) No. PD  134953 and no. 124950.  
		E-mail: {\tt zoltan.lorant.nagy@ttk.elte.hu}}
	\and Zsuzsa Weiner \thanks{Prezi.com, H-1065 Budapest, Nagymező utca 54-56, Hungary;  E-mail: {\tt zsuzsa.weiner@gmail.com}
}}
\date{}

\begin{document}
	
	\maketitle

	\begin{abstract}
		We study intersection size distributions arising in finite geometries and vector spaces.
		For a projective plane \(\Pi_q\) of order \(q\) with line set
		\(\mathcal L\), we show that
		\[
		\min_{S\subseteq\Pi_q}\max_k
		|\{\ell\in\mathcal L:|S\cap\ell|=k\}|=\Theta(q^{3/2}).
		\]
		The arguments behind these bounds admit a more general formulation.
		We prove a similar lower bound for balanced incomplete block designs, together
		with probabilistic upper bounds for uniform block systems: one for fixed
		block size under a weak dependency condition and another for growing
		block size under bounded pairwise intersections. This yields corresponding
		results for subspaces of vector spaces over finite fields, equivalently, for affine subspaces of an affine space. 
		%, furthermore for projective subspaces and for inversive planes.   In particular, for the circles of an inversive plane of order \(q\), the minimum possible maximum frequency of an intersection size is \(\Theta(q^{5/2})\). 
		We also study explicit constructions in affine
		planes of prime order and relate their intersection size distributions
		to character sum estimates; an elliptic curve construction gives a bound
		within a polylogarithmic factor of the optimal order. Finally, we relate
		balanced intersection size distributions to legitimate colorings and
		prove that every \(n\)-uniform linear hypergraph with \(n\) edges admits
		a legitimate \(2\)-coloring.
	\end{abstract}

	\section{Introduction}

	A recurring theme in combinatorics and incidence geometry is to understand how a set is distributed
	with respect to a family of test objects, such as lines, subspaces or in general, hyperedges of a hypergraph defined on the ground set. 
	%\zsst{In this paper, we consider finite point sets of finite projective planes and investigate how evenly the sizes of their intersection with the lines of the plane can be distributed.}
	{In this paper, we study how frequently the same intersection size must
		occur in finite geometries and, more generally, in uniform block systems.
		Our motivating problem concerns the intersection sizes of a point set with the lines of a finite projective plane.}
	
	%a variant of \zs{the} \zsst{ this }
	%a question { concerning the distribution of secant sizes of point sets in finite projective planes}. 
	
	% which might be of independent interest and is related to the concept of legitimate colorings (due to Alon and Füredi), to block designs induced by finite geometries and to the famous direction problem as well.
	%Instead of considering two-colorings, we simply consider sets (which might be viewed as the set of blue points) and study extremal problems concerning the distribution of the secant sizes of the sets (with respect to the size of the set, and the order of the field).
	
	Let $S$ be a finite point set in an affine or projective plane with point set $\cP$ and line set $\cL$. For each $k$, let  $L_k$ denote the set of lines  intersecting $S$ in exactly $k$ points, that is, \[
	L_k = L_k(S) := \{\ell\in \cL : |\ell \cap S| = k\}.
	\]
	
	First, we briefly discuss the Euclidean setting for comparison. Since we consider finite point sets, this is equivalent to the projective setting over the reals.
	
	\begin{proposition}\label{prop:melchior_ratio}
		Let $S$ be a finite non-collinear set of points in the real projective plane. 
		For each integer \(k\) with \(2\le k\), let
		us write
		$
		x_k := |L_k|, 
		$
		then
		\[
		\ensuremath{\frac{1}{3}<\max\left\{\frac{x_2}{ \sum_{2\le k} x_k},\frac{x_3}{ \sum_{2\le k} x_k}\right\}}.
		\]
		
	\end{proposition}
	
	\begin{proof}
		We use Melchior's inequality for real line arrangements. In its dual formulation 
		(for point sets in the real projective plane), it asserts that
		\begin{equation}\label{eq:melchior}
			\ensuremath{3 + \sum_{4\le k}(k-3)x_k\le x_2},
		\end{equation}
		see \cite{ Hirzebruch1986, Melchior1941, PachAgarwal1995}.
		Note that this in turn resolved the famous question of Sylvester and Gallai on the number of $2$-secants.
		Since \(1\le k-3\) for all \(4\le k\), it follows immediately that
		$
		\ensuremath{\sum_{4\le k}x_k\le x_2}.$
		Thus either \(x_3<x_2\) and then \(1/3<\frac{x_2}{\sum x_k}\), or $x_2\le x_3$ and then \(1/3<\frac{x_3}{\sum x_k}\).
	\end{proof}

	As a consequence, this setting always forces very strong concentration concerning the intersection sizes. 
	Note that a closely related problem is to {minimize} the number of $2$-secants, which is the subject of the  celebrated Dirac-Motzkin conjecture, resolved by Green and Tao \cite{Green-Tao}.
	
	We will point out a different kind of concentration in the finite projective setting. In order to do this, we introduce some definitions.

	\begin{defi}[$k$-secants] Let $\Pi_q$ denote a projective plane of order $q$ with point set $\cP$ and line set $\cL$.
		A line $\ell$ in $\cL$ is called a\textit{ $k$-secant of a set} $S$ if $|\ell\cap S|=k$. In the cases $k=0$ and $k=1$, these lines are usually called \textit{skew} lines and\textit{ tangent }lines (to $S$), respectively.\\  Let us denote by $n_{\ell}:=|S\cap \ell|$ the intersection size of $\ell$ with $S$, i.e., the secant size w.r.t. $\ell$,  and let $L_k(S)$ denote the set of $k$-secant lines, i.e., $L_k(S)=\{\ell\in \cL: n_{\ell}=k\}$.
	\end{defi}

	% Let $s_m(S)$ denote the cardinality of the {set of the} most common secants of $S$, i.e., $$s_m(S):=\max_k \{|L_k(S)|\}.$$
	
	%  Let $s_m(S)$ denote the maximum number of occurrences of any secant size, i.e., $$s_m(S):=\max_k \{|L_k(S)|\}.$$
	
	\begin{problem}\label{prob:main}
		Determine $$\min_{ S\subseteq \Pi_q} \max_k \{|L_k(S)|\},$$ in other words, determine the best possible lower bound on the {maximum frequency of secant sizes} for point sets of a  projective plane of order $q$.
	\end{problem}
	
	{The following theorem answers Problem~\ref{prob:main} up to a constant
		factor.}
	
	\begin{theorem}\label{leggyakoribb}
		{For any set $S \subseteq \Pi_q$, the most common secant size appears at least $\Omega(q^{3/2})$ times.} This order of magnitude is best possible, in particular,
		$$%\zsst{\ensuremath{\frac{1}{\sqrt{3}}q^{3/2}-3q}}
		{\ensuremath{\frac{q^2+q+1}{\sqrt{3q+1}}}}\le \min_{ S\subseteq \Pi_q} \max_k \{|L_k(S)|\}\le  %\zsst{\ensuremath{(1+o(q^{-\frac{1}{2}}))\sqrt{\frac{2}{\pi}}q^{3/2}}}
		{\ensuremath{\left(\sqrt{\frac{2}{\pi}}+o(1)\right)q^{3/2}}}.$$
	\end{theorem}

	\noindent Observe that the lower and upper bounds differ in a multiplicative constant smaller than $1.382$. 
	
	Theorem~\ref{leggyakoribb} will be derived from %more general lower and upper bounds for block designs developed in Section~\ref{sec:general-block-systems}, and 
	obtaining lower and upper bounds for the following generalization of Problem \ref{prob:main}.
	\begin{problem} Let $\mathcal{H}$ denote an $m$-uniform  set system on a $v$-element ground set $V$. 
		Determine $$\min_{ S\subseteq V} \max_k |\{H\in \mathcal{H}: |H\cap S|=k\}|.$$
	\end{problem}
	
	The general lower bound applies to balanced incomplete block designs, while two probabilistic upper bounds apply to uniform block systems: one for fixed block size under a weak dependency condition, and one for growing block size when pairwise intersections are uniformly bounded. We apply these results to affine and projective subspaces and to inversive planes, and obtain the following results.
	In the following theorems, ${n\brack k}_q$ stands for the Gaussian ($q$-binomial) coefficient, that is, the number of \(k\)-dimensional
	subspaces of \(\mathbb F_q^n\).
	\begin{theorem}\label{affin_result} Let $\mathcal{H}$ denote the set of  $k$-dimensional affine subspaces of the $n$-dimensional affine space $\AG(n,q)$. Fix $k$ and $q$. Then $$C q^{n-\frac{3}{2}k}{n\brack k}_q<\min_{ S\subseteq V} \max_j |\{H\in \mathcal{H}: |H\cap S|=j\}|<C' q^{n-\frac{3}{2}k}{n\brack k}_q,$$
		for some  absolute constants $C, C'>0$, independent of $n$.
	\end{theorem}
	
	\begin{theorem}\label{proj_result}
		Let $\mathcal{H}$ denote the set of $k$-dimensional subspaces of the $n$-dimensional projective space $\PG(n,q)$.  Fix $k$ and $q$. Then 
		
		$$C_1q^{-k/2} {n+1\brack k+1}_q<\min_{ S\subseteq V} \max_j |\{H\in \mathcal{H}: |H\cap S|=j\}|< C_2q^{-k/2} {n+1\brack k+1}_q,$$ for some  absolute constants $C_1, C_2>0$, independent of  $n$.
	\end{theorem}

	For an inversive (or Möbius) plane of order \(q\), we obtain respective  bounds of order \(q^{5/2}\).
	%\todo[inline]{itt valamilyen fő eredményeket még részleteznék, talán  vektortereseket}
	%\zsst{Our main problem investigates the discrepancy of the secant size distribution; more precisely, it measures how close the secant size distribution can be to the uniform distribution over $[0, q+1]$.}
	{The maximum frequency studied in Problem~\ref{prob:main} measures the
		balance of the secant size distribution through the size of its largest
		class.}
	The distribution of secant sizes is closely related to various central problems in finite geometry, or in a broader context, combinatorics. The most notable examples include the characterizations of blocking sets or $k$-fold blocking sets, where $0$-secants, and all $t$-secants for $t<k$ are forbidden, respectively, and  $(k, n)$-arcs, where   all $t$-secants for $t>k$ are forbidden. Various  configurations can be characterized by having only \textit{very few} secant sizes, such as subplanes or unitals, which have many important applications in extremal graph theory or Ramsey theory \cite{Mubayi-Vers, Sam-Vers}. Even avoiding a single secant size can be challenging, we refer to the theory of untouchable (or tangent-free) sets \cite{Blokhuis, BSzW} or \cite{Heger-Nagy} for the general case. In a sequel paper, we also discuss the largest possible number of secants with the same size, in terms of the size of the set \cite{ZsZ2}. 
	The conceptual framework of studying intersection distribution also appeared in connection with characterizing polynomials over finite fields by analyzing the intersection patterns between polynomial graphs and affine lines in the work of Li and Pott \cite{Li-Pott}. They introduced the  "non-hitting index", which simply enumerates the affine lines that do not intersect a given polynomial's graph $S$; thus it coincides with $|L_0(S)|$. This index has proven vital for  deriving bounds related to  Kakeya sets in affine planes \cite{Li-Pott}. More recently, their approach was further investigated in \cite{Huczynska}.
	Intersection-size questions for affine subspaces have also been studied
	from several related extremal viewpoints.  Kov\'acs and Nagy
	\cite{KovacsNagy25} investigated which intersection sizes with
	\(k\)-dimensional affine subspaces of \(AG(n,2)\) are unavoidable for
	point sets of prescribed cardinality, and also studied constructions
	whose set of possible intersection sizes is small.  Thus their problem
	concerns essentially the support of the intersection-size distribution,
	whereas here we study the multiplicities of its values.
	Very recently, Xu \cite{Xu26} introduced affine subspace statistics in
	\(\mathbb F_2^n\), studying, for a fixed intersection size \(s\), the
	maximum proportion of \(d\)-flats satisfying \(|A\cap F|=s\).
	Subsequent work of Chao, Xu and Zakharov \cite{CXZ26} obtained
	asymptotically sharp bounds for this parameter in several regimes,
	including intersection sizes with prescribed \(2\)-adic valuation, and
	determined the case \(s=1\) exactly.  These problems are complementary
	to the present one: they maximize the frequency of a prescribed
	intersection size, whereas we minimize the largest frequency over all
	intersection sizes.
	%The characterization of sets with no skew lines, or in general, sets without small secants is the subject of the study of covering and $(k-)$fold blocking sets \cite{}. Sets having line intersection of only  few different sizes is a vivid topic in finite geometry \cite{} which is closely related to block designs \cite{}, and notable examples are  subplanes, complete arcs, $(k;n)$-arcs and Korchmáros-Mazzocca arcs \cite{}. In fact, avoiding a single secant size can raise challenges, we refer to the theory of untouchable (or tangent-free) sets \cite{ , } or \cite{Heger-Nagy} for the general case.\\

	%Discrepancy problems study how evenly the elements of a finite set system can be divided under a two--coloring.
	%Given a finite set $X$ and a family $\mathcal{F}$ of subsets of $X$, the discrepancy of a coloring
	%$\chi : X \to \{-1,+1\}$ is defined as
	%\[
	%\mathrm{disc}(\mathcal{F},\chi)
	% = \max_{F \in \mathcal{F}} \left| \sum_{x \in F} \chi(x) \right|,
	%\]
	%and the discrepancy of $\mathcal{F}$ is the minimum of this quantity over all such colorings.
	%The central goal is to bound this minimum in terms of structural parameters of the set system. 

	Determining the smallest possible maximum frequency of secant sizes  is also related to the concept of legitimate colorings, introduced by Rosa, and studied by  Alon and Füredi \cite{Alon-Furedi}. A\textit{ legitimate coloring of a hypergraph} is a coloring of the vertex set by $k$ colors such that all hyperedges are distinguishable by their color distribution, in other words, by their multiplicity list $c_1(H), c_2(H), \ldots, c_k(H)$, where $c_i(H)$ denotes the number of vertices in hyperedge $H$ colored with the $i$-th color. The minimum number of colors for which a legitimate coloring of the hypergraph exists is the \textit{legitimate coloring number}. Alon and Füredi initiated the study of this coloring number by considering the hypergraph defined by the lines of a projective plane.  {Observe that a natural obstacle to having a small legitimate coloring number is when the distribution of $c_i$ is very unbalanced in every coloring. This corresponds exactly to the distribution of secant sizes when one considers  projective planes. We will revisit this link in Section \ref{sec:appl}.}%\todo{ez nem csak 2-színezésnél gond}
	%\zs{For a \(2\)-coloring, one color class can be identified with a subset \(S\) of the vertex set, and its multiplicity in a hyperedge \(H\) is exactly \(|H\cap S|\). Thus a high frequency of one intersection size is an obstruction to a legitimate \(2\)-coloring. We revisit this connection in Section~\ref{sec:appl}, where we study legitimate colorings of linear hypergraphs.}

	The celebrated direction problem \cite{Ball03, BBB99} is also closely related to the question posed above. Indeed, 
	the non-determined directions of a point set in $\Pi_q$ of size $q$ correspond to slopes on which the point set is equidistributed. For the case of point sets of {arbitrary} size, we refer to
	\cite{Ghidelli, KS24}.
	\begin{comment}
	\zsst{Our paper is organised as follows. In Section 2 we prove our main result, Theorem 1.4. To prove the lower bound, we point out how strongly the standard equations influence the distribution of secant sizes, preventing it from being very uniform. The upper bound relies on a randomised construction. Section 3 is devoted to the presentation of explicit constructions using low-degree polynomials: parabolas and elliptic curves. These point sets will not be extremal w.r.t. $\min_{ S\subseteq \Pi_q} \max_k \{|L_k(S)|\}$, but are expected to be not far off, and have links to intensively studied problems in number theory, namely character sum estimates. In Section~\ref{sec:appl}, we discuss the connection between legitimate
	colorings of \(\PG(2,q)\) and our main result, and prove a statement that
	can be viewed as a legitimate-coloring version of the
	Erdős--Faber--Lovász conjecture. Finally,
	Section~\ref{sec:general-block-systems} develops the general block-system
	framework described above and presents its finite-geometric applications
	and refinements for small parameters.}
	\end{comment}
	{Our paper is organised as follows. Section~\ref{sec:general-block-systems} develops the general  framework to minimize maximum frequency of intersection sizes for block design structures. Then in Section \ref{sec:application}, we  derive Theorem~\ref{leggyakoribb} for projective planes, and continue by obtaining the respective results for $k$-dimensional subspaces of vector spaces and projective spaces, and for inversive (or Möbius) planes.  Section \ref{sec:ref} presents refinements for the affine and projective space case, when the subspace has small cardinality. Then in Section \ref{sec:constructions in affine planes}, we give three explicit constructions in affine planes of prime
		order, based on parabolas and elliptic curves, and relate their
		intersection size distributions to character sum estimates. These
		constructions also yield projective-plane examples through the standard
		embedding of the affine plane. In Section~\ref{sec:appl}, we relate the
		problem to legitimate colorings and prove that every \(n\)-uniform linear
		hypergraph with \(n\) edges admits a legitimate \(2\)-coloring. An alternative elementary lower-bound argument for projective planes is included in the appendix.}

	\section[Bounds on the maximum frequency of an intersection size in block systems]{{Bounds on the maximum frequency of an intersection size in block designs}}
	\label{sec:general-block-systems}
	{Problem~\ref{prob:main} admits a natural formulation in the more general setting of block designs.}
	%\todo{system$\to$ design cserét csináltam}
	
	Let \(V\) be a finite set of size \(v\), and let \(\mathcal H\) be a family of subsets of \(V\), called blocks. We write \(N=|\mathcal H|\). The block system \((V,\mathcal H)\) is \(m\)-uniform if every block has size \(m\).
	
	An \(m\)-uniform block system is a \(2\text{-}(v,m,\lambda)\) design if every pair of distinct points lies in exactly \(\lambda\) blocks. If \(m<v\), it is a balanced incomplete block design. In this case every point lies in the same number \(r\) of blocks. Counting point--block incidences in two ways, we get
	\[
	vr=Nm,
	\qquad
	\frac{r}{N}=\frac{m}{v},
	\]
	and counting pairs consisting of two distinct points and a block containing them, we obtain
	\[
	\binom{v}{2}\lambda=N\binom{m}{2},
	\qquad
	\frac{\lambda}{N}=\frac{m(m-1)}{v(v-1)}.
	\]
	
	For \(S\subseteq V\), define
	\[
	\mathcal H(S,i)=\{H\in\mathcal H:|H\cap S|=i\},
	\qquad 0\leq i\leq m,
	\]
	and put \[M_{\mathcal H}(S)=\max_{0\leq i\leq m}|\mathcal H(S,i)|.\] Finally, define the minimum of $M_{\mathcal H}(S)$  over all sets.
	
	\begin{defi}
		\[B(\mathcal H):=\min_{S\subseteq V}M_{\mathcal H}(S).\]
	\end{defi}

	{Our aim is to establish lower and upper bounds for \(B(\mathcal H)\) under suitable assumptions on the block system \((V,\mathcal H)\). Applying these bounds to the lines of a projective plane will yield Theorem~\ref{leggyakoribb}.}
	
	\subsection[General lower and upper bounds]{{General lower and upper bounds}}
	For the lower bound, we assume that the block system is a balanced
	incomplete block design. The proof relies on the following inequality
	for integer-valued random variables. 
	
	\begin{lemma}[Bobkov--Marsiglietti--Melbourne
		{\cite[Theorem~1.1]{BMM}}]
		\label{lem:maximum-atom}
		If \(X\) is integer-valued, then
		\[
		\ensuremath{\frac{1}{12}\left(\frac{1}{\max_j\Pr(X=j)^2}-1\right)
			\leq\operatorname{Var}(X)}.
		\]
	\end{lemma}
	Note that we only use their lower bound, so some of their further assumptions can be omitted, as they also observe in \cite {BMM}.
	
	We now separate the lower and upper bounds, since they require different assumptions.
	
	\begin{theorem}
		\label{thm:design-lower}
		Let \((V,\mathcal H)\) be a balanced incomplete
		\(2\text{-}(v,m,\lambda)\) design. Then
		\[
		\frac{N}{\sqrt{3m+1}}\leq \frac{N}{\sqrt{3m\frac{v-m}{v-1}+1}}\leq B(\mathcal H).
		\]
	\end{theorem}
	
	\begin{proof}
		Fix \(S\subseteq V\), choose \(H\in\mathcal H\) uniformly at random,
		and let \(X=|H\cap S|\). The \(2\)-design identities give
		\[
		\mathbb E[X]=\frac{m|S|}{v},
		\qquad
		\mathbb E[X(X-1)]
		=\frac{m(m-1)|S|(|S|-1)}{v(v-1)}.
		\]
		Indeed, each point of \(S\) belongs to \(H\) with probability
		\(r/N=m/v\), while \(X(X-1)\) counts ordered pairs of distinct
		points of \(S\) contained in \(H\), each of which occurs with probability
		$$\lambda/N=\frac{m(m-1)}{v(v-1)}.$$ Therefore
		\[
		\begin{aligned}
			\operatorname{Var}(X)
			&=\mathbb E[X(X-1)]+\mathbb E[X]-\mathbb E[X]^2\\
			&=m\frac{|S|}{v}\left(1-\frac{|S|}{v}\right)
			\frac{v-m}{v-1}
			\leq\frac{m}{4}\frac{v-m}{v-1}.
		\end{aligned}
		\]
		Since \(\max_i\Pr(X=i)=M_{\mathcal H}(S)/N\),
		Lemma~\ref{lem:maximum-atom} yields
		\[
		\frac{N}{\sqrt{3m\frac{v-m}{v-1}+1}}
		\leq\frac{N}{\sqrt{12\operatorname{Var}(X)+1}} \leq M_{\mathcal H}(S)
		.
		\]
		Taking the minimum over \(S\subseteq V\) proves the lower bound.
	\end{proof}
	
	For the upper bound, the design property is not needed; for fixed block size it is enough to assume the following weak dependency condition.
	
	\begin{theorem}
		\label{thm:fixed-upper}
		\label{thm:general-hypergraph}
		Let \(m\) be fixed, and let \((V_n,\mathcal H_n)\) be a sequence of
		\(m\)-uniform block systems, with \(N_n=|\mathcal H_n|\). Suppose that
		\[
		D_n:=|\{(H,H')\in\mathcal H_n^2:H\cap H'\neq\emptyset\}|
		=o(N_n^2)
		\qquad\text{as }n\to\infty.
		\]
		Then
		\[
		B(\mathcal H_n)\leq (\rho_m+o(1))N_n,
		\qquad
		\rho_m=\inf_{0\leq p\leq1}\max_{0\leq i\leq m}
		\binom mi p^i(1-p)^{m-i}.
		\]
	\end{theorem}
	
	\begin{proof}
		Fix \(p\in[0,1]\) and choose \(S\subseteq V_n\) by
		including each point independently with probability \(p\). For every
		\(0\leq i\leq m\),
		\[
		\mathbb E\bigl[|\mathcal H_n(S,i)|\bigr]
		=N_n\binom mi p^i(1-p)^{m-i}.
		\]
		Since
		\[
		|\mathcal H_n(S,i)|
		=\sum_{H\in\mathcal H_n}\mathbf 1_{\{|H\cap S|=i\}},
		\]
		its variance is the sum of the covariances of these indicators. Indicators
		corresponding to disjoint blocks are independent, so only intersecting
		ordered pairs contribute. Since each covariance has absolute value at most
		\(1\), we obtain
		\[
		\operatorname{Var}\bigl(|\mathcal H_n(S,i)|\bigr)
		\leq D_n=o(N_n^2).
		\]
		
		{For each fixed \(i\) and every \(\varepsilon>0\), Chebyshev's inequality gives
			\[
			\Pr\left(
			\left|
			\frac{|\mathcal H_n(S,i)|}{N_n}
			-\binom mi p^i(1-p)^{m-i}
			\right|>\varepsilon
			\right)
			\leq
			\frac{\operatorname{Var}(|\mathcal H_n(S,i)|)}
			{\varepsilon^2N_n^2}
			=o(1).
			\]
			Thus \(|\mathcal H_n(S,i)|/N_n\) differs from its expectation by \(o(1)\)
			with probability tending to \(1\). Since \(m\) is fixed, there are only
			\(m+1\) possible values of \(i\), so the union bound shows that these
			estimates hold simultaneously. Thus, with probability tending to \(1\),
			for every \(0\leq i\leq m\),
			\[
			\frac{|\mathcal H_n(S,i)|}{N_n}
			=
			\binom mi p^i(1-p)^{m-i}+o(1).
			\]
			Consequently, there exists \(S\subseteq V_n\) such that
			\begin{equation}\label{eq:upper_M}
				M_{\mathcal H_n}(S)
				\leq
				\left(
				\max_{0\leq i\leq m}
				\binom mi p^i(1-p)^{m-i}
				+o(1)
				\right)N_n.    
			\end{equation}

			Taking the infimum over \(p\in[0,1]\) proves the upper bound.}
	\end{proof}
	
	\begin{remark}[The constant in the upper bound \eqref{eq:upper_M} of $M_{\mathcal H_n}$]
		For \(0\leq i\leq m\), write
		\[
		b_i(p)=\binom mi p^i(1-p)^{m-i}.
		\]
		For fixed \(0<p<1\),
		\[
		\frac{b_{i+1}(p)}{b_i(p)}
		=
		\frac{m-i}{i+1}\frac{p}{1-p},
		\]
		and this ratio decreases with \(i\). 
		
		\begin{comment}
		
		Hence the terms can change from
		increasing to decreasing at most once. For \(1\leq j\leq m-1\), it
		follows that \(b_j(p)\) is maximal precisely when it is at least as large
		as its two neighbours. Using the ratio above,
		\[
		b_{j-1}(p)\leq b_j(p)
		\iff p\geq\frac{j}{m+1},
		\qquad
		b_{j+1}(p)\leq b_j(p)
		\iff p\leq\frac{j+1}{m+1}.
		\]
		For \(j=0\) and \(j=m\), the same conclusion follows from the single
		neighbouring inequality.
		Thus, for every \(0\leq j\leq m\), \(j\) is a
		maximizing index precisely when
		\[
		\frac{j}{m+1}\leq p\leq\frac{j+1}{m+1}.
		\]
		At a common endpoint of two such intervals, the corresponding consecutive
		terms are equal.
		On the interval corresponding to \(j\), we have
		\(\max_i b_i(p)=b_j(p)\). Direct differentiation shows that the only
		possible interior critical point of \(b_j(p)\) is \(p=j/m\), where it
		attains a maximum. Hence its minimum on this interval is attained at an
		endpoint.
		\end{comment}
		This yields that the minimum of the expression  over all \(p\) is attained at
		\[
		p=\frac{j+1}{m+1},
		\qquad 0\leq j\leq m-1,
		\]
		where \(b_j(p)=b_{j+1}(p)\). Substitution gives
		\begin{equation}\label{eq:rho}
			\rho_m=
			\min_{0\leq j\leq m-1}
			\binom mj
			\left(\frac{j+1}{m+1}\right)^j
			\left(\frac{m-j}{m+1}\right)^{m-j}.    
		\end{equation}

		The values in this minimum are symmetric about the central indices, and
		comparison of consecutive values shows that they decrease towards the
		centre. Thus it attains its minimum at \(j=m/2+O(1)\), and
		Stirling's formula gives
		\[
		\rho_m=
		\left(\sqrt{\frac{2}{\pi}}+o(1)\right)m^{-1/2}
		\qquad\text{as }m\to\infty.
		\] 
	\end{remark}
	
	{For growing block sizes, we obtain an upper bound when distinct blocks have uniformly bounded intersections.}
	
	\begin{theorem}
		\label{thm:growing-upper}
		{Let \((V_n,\mathcal H_n)\) be a sequence of \(m_n\)-uniform block systems,
			with \(N_n=|\mathcal H_n|\). Suppose that \(m_n\to\infty\),
			\[
			\max_{\substack{H,H'\in\mathcal H_n\\H\neq H'}}|H\cap H'|=O(1)
			\]
			and
			\[
			m_n\sqrt{\log m_n}=o(N_n).
			\]
			Then
			\[
			B(\mathcal H_n)\leq
			\left(\sqrt{\frac{2}{\pi}}+o(1)\right)
			\frac{N_n}{\sqrt{m_n}}.
			\]}
	\end{theorem}
	
	\begin{proof} For the upper bound, choose \(S\subseteq V_n\) by including each point independently with probability \(1/2\). The proof follows the same strategy as that of Theorem~\ref{thm:fixed-upper}, but the growing block size requires a sharper covariance estimate in the central range and a separate treatment of the binomial tails. For every \(0\leq i\leq m_n\), a fixed block meets \(S\) in exactly \(i\) points with probability \(\binom{m_n}{i}2^{-m_n}\). Therefore,
		\[
		\mathbb E\bigl[|\mathcal H_n(S,i)|\bigr]
		=N_n\binom{m_n}{i}2^{-m_n}.
		\]
		Moreover,
		\[
		|\mathcal H_n(S,i)|
		=\sum_{H\in\mathcal H_n}\mathbf 1_{\{|H\cap S|=i\}},
		\]
		so its variance is the sum of the covariances of these indicators.
		
		We first consider the central values satisfying
		\(\left|i-m_n/2\right|\leq\sqrt{m_n\log m_n}\). Let \(H\) and \(H'\)
		be distinct blocks and write \(t=|H\cap H'|\). If exactly \(j\) points
		of \(H\cap H'\) belong to \(S\), then the conditional probability that
		\(H\) meets \(S\) in \(i\) points is
		\(\binom{m_n-t}{i-j}2^{-(m_n-t)}\). Since \(t=O(1)\),
		\(0\leq j\leq t\), and \(i\) lies in the central range,
		\[
		\frac{\binom{m_n-t}{i-j}2^{-(m_n-t)}}
		{\binom{m_n}{i}2^{-m_n}}
		=
		2^t\frac{\binom{m_n-t}{i-j}}{\binom{m_n}{i}}
		=
		1+O\left(\sqrt{\frac{\log m_n}{m_n}}\right),
		\]
		where the last estimate follows by writing
		\(i=m_n/2+O(\sqrt{m_n\log m_n})\) and expressing the ratio of the
		binomial coefficients as a product of \(t=O(1)\) factors. Once the
		choices in \(H\cap H'\) are fixed, the two events are independent.
		Their covariance is therefore the variance of the conditional probability
		above. Its relative deviation from
		\(\binom{m_n}{i}2^{-m_n}\) is
		\(O(\sqrt{\log m_n/m_n})\), and hence, uniformly for distinct \(H,H'\),
		\[
		\operatorname{Cov}\left(
		\mathbf 1_{\{|H\cap S|=i\}},
		\mathbf 1_{\{|H'\cap S|=i\}}
		\right)
		=
		O\left(
		\left(\binom{m_n}{i}2^{-m_n}\right)^2
		\frac{\log m_n}{m_n}
		\right)
		=
		O\left(\frac{\log m_n}{m_n^2}\right),
		\]
		where we used
		\(\max_i\binom{m_n}{i}2^{-m_n}=O(m_n^{-1/2})\). The diagonal terms in the variance sum contribute at most \(N_n\binom{m_n}{i}2^{-m_n}=O(N_n/\sqrt{m_n})\), while there are at most \(N_n^2\) off-diagonal terms. Hence
		\[
		\operatorname{Var}\bigl(|\mathcal H_n(S,i)|\bigr)
		=
		O\left(
		\frac{N_n}{\sqrt{m_n}}
		+\frac{N_n^2\log m_n}{m_n^2}
		\right)
		\]
		uniformly in the central range.
		
		There are \(O(\sqrt{m_n\log m_n})\) central values of \(i\). Chebyshev's inequality bounds the probability of a large deviation for each such \(i\), and the union bound adds these probabilities. The variance estimate above and the assumption \(m_n\sqrt{\log m_n}=o(N_n)\) therefore show that, with probability tending to \(1\), simultaneously for every central \(i\),
		\[
		|\mathcal H_n(S,i)|
		=
		N_n\binom{m_n}{i}2^{-m_n}
		+o\left(\frac{N_n}{\sqrt{m_n}}\right).
		\]
		
		It remains to consider the values outside the central range. For every block \(H\), the binomial tail bound gives
		\[
		\Pr\left(
		\left||H\cap S|-\frac{m_n}{2}\right|
		>\sqrt{m_n\log m_n}
		\right)
		\leq 2m_n^{-2}.
		\]
		Thus the expected total number of blocks whose intersection size lies outside this range is at most \(2N_n/m_n^2\). Markov's inequality shows that, with probability tending to \(1\), this number is at most \(N_n/m_n=o(N_n/\sqrt{m_n})\).
		
		Combining the central and tail estimates, there exists \(S\subseteq V_n\) such that
		\[
		M_{\mathcal H_n}(S)
		\leq
		N_n\max_{0\leq i\leq m_n}
		\left\{\binom{m_n}{i}2^{-m_n}\right\}
		+o\left(\frac{N_n}{\sqrt{m_n}}\right).
		\]
		Finally, Stirling's formula gives
		\[
		\max_{0\leq i\leq m_n}
		\left\{\binom{m_n}{i}2^{-m_n}\right\}
		=
		\left(\sqrt{\frac{2}{\pi}}+o(1)\right)m_n^{-1/2},
		\]
		which proves the result.
	\end{proof}
	\section{Applications: affine and projective spaces, inversive planes}
	\label{sec:application}
	\subsection[Projective planes]{{Projective planes}}
	
	{We first derive Theorem~\ref{leggyakoribb} from the general bounds.}
	
	\begin{proof}[{Proof of Theorem~\ref{leggyakoribb}}]
		{Let \(N=q^2+q+1\). The lines of \(\Pi_q\) form a balanced incomplete \(2\text{-}(N,q+1,1)\) design. Moreover,}
		\[
		{\ensuremath{(q+1)\frac{N-q-1}{N-1}=q.}}
		\]
		{Hence Theorem~\ref{thm:design-lower} gives}
		\[
		{\ensuremath{\frac{N}{\sqrt{3q+1}}\leq B(\mathcal L).}}
		\]
		{For the upper bound, apply Theorem~\ref{thm:growing-upper} to the lines of \(\Pi_q\). Here \(m_q=q+1\), \(N_q=q^2+q+1\), any two distinct lines intersect in one point, and \(m_q\sqrt{\log m_q}=o(N_q)\). Hence}
		\[
		{\ensuremath{B(\mathcal L)\leq
				\left(\sqrt{\frac{2}{\pi}}+o(1)\right)
				\frac{q^2+q+1}{\sqrt{q+1}}
				=\left(\sqrt{\frac{2}{\pi}}+o(1)\right)q^{3/2}.}}
		\]
		{This completes the proof.}
	\end{proof}
	{The variance computation underlying Theorem~\ref{thm:design-lower} also gives a more precise lower bound depending on the size of \(S\).}
	
	\begin{proposition}\label{prop:projective-lower}
		{Let \(\Pi_q\) be a projective plane of order \(q\), let \(N=q^2+q+1\), and let \(S\) be a set of points in \(\Pi_q\). Then}
		\[
		{\ensuremath{\frac{N^{3/2}}
				{\sqrt{12q|S|\left(1-\frac{|S|}{N}\right)+N}}
				\leq\max_k|L_k(S)|.}}
		\]
	\end{proposition}
	
	\begin{proof}
		{Choose a line uniformly at random and let \(X\) be its intersection size with \(S\). Since the lines form a \(2\text{-}(N,q+1,1)\) design, the variance computation in the proof of Theorem~\ref{thm:design-lower} gives}
		\[
		{\ensuremath{\operatorname{Var}(X)
				=q\frac{|S|}{N}\left(1-\frac{|S|}{N}\right).}}
		\]
		{Moreover, \(\max_k\Pr(X=k)=\max_k|L_k(S)|/N\), so Lemma~\ref{lem:maximum-atom} gives the result.}
	\end{proof}

	\subsection{Vector spaces, Affine spaces}
	For $\AG(n,q)$, let \(\mathcal H_n\) be the family of affine $k$-dimensional subspaces, or in other words, 
	\(k\)-flats.  Then
	$        v_n=q^n,  $   $  m=q^k,$
	and
	\[
	N_n
	=
	q^{n-k}{n\brack k}_q.
	\]
	
	The affine \(k\)-flats form a \(2\)-design: every point lies in
	\({n\brack k}_q\) affine \(k\)-flats, and every pair of distinct points
	lies in \({n-1\brack k-1}_q\) affine \(k\)-flats.  Moreover, the weak
	dependency condition holds.  Indeed, for a fixed \(k\)-flat \(F\), the
	number of \(k\)-flats meeting \(F\) is at most
	\[
	m{n\brack k}_q,
	\]
	because one may first choose a point of \(F\), and then choose a
	\(k\)-flat through that point.  Hence
	\[
	D_n
	\leq
	N_n m{n\brack k}_q
	=
	O(q^{2k-n})N_n^2
	=
	o(N_n^2),
	\]
	for fixed \(q\) and fixed \(k\). Since the block size \(m=q^k\) is then also
	fixed, 
	Theorems~\ref{thm:design-lower} and~\ref{thm:fixed-upper} give
	\[
	\frac{N_n}
	{\sqrt{3q^k\frac{q^n-q^k}{q^n-1}+1}}
	\leq
	B(\AG(n,q),k)
	\leq
	\left(\rho_{q^k}+o(1)\right)N_n.
	\]
	In particular,
	\[
	B(\AG(n,q),k)
	=
	\Theta(     q^{n-\frac{3}{2}k}{n\brack k}_q), 
	\]
	which proves Theorem \ref{affin_result}.
	
	\subsection{Projective spaces} For $\PG(n,q)$, let \(\mathcal H_n\) be the family of projective
	\(k\)-subspaces.  Then
	\[
	v_n=\frac{q^{n+1}-1}{q-1},
	\qquad
	m=\frac{q^{k+1}-1}{q-1},
	\]
	and
	\[
	N_n={n+1\brack k+1}_q.
	\]
	Again the projective \(k\)-subspaces form a \(2\)-design: every point
	lies in \({n\brack k}_q\) such subspaces, and every pair of distinct
	points lies in \({n-1\brack k-1}_q\) such subspaces.  The same
	intersection-counting argument gives
	\[
	D_n
	\leq
	N_n m{n\brack k}_q
	=
	o(N_n^2),
	\]
	for fixed \(q\) and fixed \(k\). Since the block size
	\(
	m=\frac{q^{k+1}-1}{q-1}
	\)
	is then also fixed,
	Theorems~\ref{thm:design-lower} and~\ref{thm:fixed-upper} give
	\[
	\frac{N_n}
	{\sqrt{3m\frac{v_n-m}{v_n-1}+1}}
	\leq
	B(\PG(n,q),k)
	\leq
	\left(\rho_m+o(1)\right)N_n,
	\]
	where \(m=(q^{k+1}-1)/(q-1)\).

	Consequently, 
	\[
	B(\PG(n,q),k)=\Theta(q^{-k/2} {n+1\brack k+1}_q).
	\]
	More precisely, when \(m\) is fixed,
	\[
	\frac{1}{\sqrt{3m+1}}N_n-o(N_n)
	\leq
	B(\PG(n,q),k)
	\leq
	\left(\rho_m+o(1)\right)N_n,
	\] which proves Theorem \ref{proj_result}.
	
	\subsection{Inversive planes}
	
	An inversive or Möbius plane of order \(q\) is an incidence structure
	\(\mathcal I=(V,\mathcal C)\) with \(|V|=q^2+1\), whose blocks, called
	circles, have size \(q+1\), and in which every three distinct points are
	contained in exactly one circle. Thus \(\mathcal I\) is a
	\(3\text{-}(q^2+1,q+1,1)\) design and hence also a
	\(2\text{-}(q^2+1,q+1,q+1)\) design. In particular,
	\[
	v=q^2+1,\qquad m=q+1,\qquad N=|\mathcal C|=q(q^2+1).
	\]
	Therefore Theorem~\ref{thm:design-lower} gives
	\[
	\frac{q(q^2+1)}{\sqrt{3q+1-\frac{3}{q}}} \leq B(\mathcal C).
	\]
	
	Unlike in the affine and projective space applications above, here the
	block size \(m=q+1\) tends to infinity as \(q\to\infty\).
	Moreover, two distinct circles meet in at most two points. Thus
	Theorem~\ref{thm:growing-upper} applies as \(q\to\infty\), since
	\[
	(q+1)\sqrt{\log(q+1)}=o\bigl(q(q^2+1)\bigr),
	\]
	and yields
	\[
	B(\mathcal C)\leq
	\left(\sqrt{\frac{2}{\pi}}+o(1)\right)
	\frac{q(q^2+1)}{\sqrt{q+1}}.
	\]
	Consequently,
	\[
	\left(\frac{1}{\sqrt3}+o(1)\right)q^{5/2}
	\leq B(\mathcal C)\leq
	\left(\sqrt{\frac{2}{\pi}}+o(1)\right)q^{5/2}.
	\]

	\section[Refinements for small parameters $m$]{{Refinements} for small parameters}
	\label{sec:ref}
	%\todo{ha elfogadjuk az egyszerusiteseket akkor frissiteni kell a tablazatot is}
	We now examine the bounds obtained by Theorems~\ref{thm:design-lower} and~\ref{thm:fixed-upper}
	for small values of the field order \(q\) and the subspace dimension
	\(k\). Throughout the remainder of this section, \(q\) and \(k\) are
	fixed and \(n\to\infty\). Table~\ref{tab:small-parameters} records the
	general bounds for \(B(\mathcal H_n)/N_n\), together with the improvements on the respective lower bound, established
	below. The bounds supplied by these two theorems depend
	on \(q\) and \(k\) only through the block size \(m\).
	
	\begin{table}[ht]
		\centering
		\small
		\setlength{\tabcolsep}{9pt}
		\begin{tabular}{lccccc}
			\hline
			space and subspaces & \(m\) & general lower & general upper, $\rho$ & refinement on the lower bound  & exact?\\
			\hline
			\(\AG(n,2)\), lines
			& \(2\)
			& \(1/\sqrt7\approx0.378\)
			& \(4/9\approx0.444\)
			&  \(4/9\approx0.444\)& \checkmark \\
			
			\(\AG(n,3)\), lines
			& \(3\)
			& \(1/\sqrt{10}\approx0.316\)
			& \(3/8=0.375\)
			&  \(1/3+\nu_{\rm FL}\approx0.333334\) & -- \\
			
			\(\PG(n,2)\), lines
			& \(3\)
			& \(1/\sqrt{10}\approx0.316\)
			& \(3/8=0.375\)
			&  \(1/3\approx0.333\) & \checkmark \\
			
			\(\AG(n,2)\), planes
			& \(4\)
			& \(1/\sqrt{13}\approx0.277\)
			& \(216/625\approx0.346\)
			&  \(\beta_4\approx0.324\) & --\\
			
			\(\PG(n,3)\), lines
			& \(4\)
			& \(1/\sqrt{13}\approx0.277\)
			& \(216/625\approx0.346\)
			& -- & --\\
			
			\(\PG(n,2)\), planes
			& \(7\)
			& \(1/\sqrt{22}\approx0.213\)
			& \(35/128\approx0.273\)
			& -- & -- \\
			
			\(\AG(n,2)\), \(3\)-flats
			& \(8\)
			& \(1/5=0.200\)
			& \(\approx0.260\)
			&  \(\beta_8\approx0.209\) & --\\
			
			\(\AG(n,3)\), planes
			& \(9\)
			& \(1/\sqrt{28}\approx0.189\)
			& \(63/256\approx0.246\)
			& -- & -- \\
			\hline
		\end{tabular}
		\caption{General and improved asymptotic bounds for small parameter
			values. Each entry is the coefficient of \(N_n\).}
		\label{tab:small-parameters}
	\end{table}

	The nontrivial entries in the second-to-last column use additional information not
	captured by the general variance argument. We discuss these refinements
	separately.
	
	{In the arguments below, for \(S\subseteq V_n\) put
		\(\alpha=|S|/v_n\), and let \(x_i\) denote the proportion of blocks
		meeting \(S\) in exactly \(i\) points. For a \(2\text{-}(v_n,m,\lambda)\)
		design, the following moment equations hold:
		\[
		\sum_{i=0}^m x_i=1,\qquad
		\sum_{i=0}^m i x_i=m\alpha,\qquad
		\sum_{i=0}^m i(i-1)x_i
		=m(m-1)\frac{|S|(|S|-1)}{v_n(v_n-1)}
		=m(m-1)\alpha^2+o(1).
		\]
		The first equation says that the proportions sum to \(1\), while the
		second and third follow by double counting incidences with points of
		\(S\) and ordered pairs of distinct points of \(S\), respectively.}

	\subsection{Affine subspaces over
		\texorpdfstring{\(\mathbb F_2\)}{F2}}
	
	Let \(S\subseteq\mathbb F_2^n\) have density \(\alpha\) {(i.e. \(\alpha=|S|/2^n\))}, and let \(x_i\)
	denote the proportion of affine \(k\)-flats meeting \(S\) in exactly
	\(i\) points. 
	
	For \(k=1\), affine lines are simply unordered pairs of
	points, and hence
	\[
	x_0=(1-\alpha)^2+o(1),\qquad
	x_1=2\alpha(1-\alpha)+o(1),\qquad
	x_2=\alpha^2+o(1).
	\]
	{The expression is unchanged when \(\alpha\) is replaced by
		\(1-\alpha\), so we may assume \(\alpha\leq1/2\). In this range,
		\((1-\alpha)^2\geq\alpha^2\), while the decreasing term
		\((1-\alpha)^2\) meets the increasing term \(2\alpha(1-\alpha)\) at
		\(\alpha=1/3\). Hence}
	\[
	\min_{0\leq\alpha\leq1}
	\max\{(1-\alpha)^2,2\alpha(1-\alpha),\alpha^2\}=\frac49.
	\]
	The minimum is attained at \(\alpha=1/3\) and \(\alpha=2/3\), so
	\[
	B(\AG(n,2),1)=\left(\frac49+o(1)\right)N_n.
	\]
	
	{For \( k\ge 2\), we obtain further improvements by counting
		\(k\)-flats contained entirely in a set. Let
		\(A\subseteq\mathbb F_2^n\) with \(\alpha=|A|/2^n\).
		The
		proportion of affine \(k\)-flats contained in \(A\) is at least
		\(\alpha^{2^k}-o(1)\) which follows from the well-known cube supersaturation result of Gowers \cite[Lemma 3.7]{Gowers01}.} Applying this to \(S\) and its complement, and
	writing \(m=2^k\), gives
	\[
	\ensuremath{\alpha^m-o(1)\leq x_m,\qquad
		(1-\alpha)^m-o(1)\leq x_0}.
	\]
	Together with the standard moment equations
	\[
	\sum_{i=0}^m x_i=1,\qquad
	\sum_{i=0}^m i x_i=m\alpha+o(1),\qquad
	\sum_{i=0}^m i(i-1)x_i=m(m-1)\alpha^2+o(1),
	\]
	{for each fixed \(\alpha\),}
	these inequalities give a finite linear program. Let \(\beta_m\) be the
	minimum possible value of \(\max_i x_i\) over all
	\(0\leq\alpha\leq1\) and all nonnegative \(x_0,\ldots,x_m\) satisfying
	\[
	\sum_i x_i=1,\qquad
	\sum_i i x_i=m\alpha,\qquad
	\sum_i i(i-1)x_i=m(m-1)\alpha^2,
	\]
	and
	\[
	\ensuremath{(1-\alpha)^m\leq x_0,\qquad \alpha^m\leq x_m}.
	\]
	Then
	\[
	\ensuremath{\beta_m\leq\liminf_{n\to\infty}\frac{B(\AG(n,2),k)}{N_n}}.
	\]
	Numerically, via linear programming, one can obtain  %\todo{ez igy szerintem nagyon hirtelen jon, most akkor az olvasonak fejben meg kene oldania ezt a lin. pr. feladatot? :)}
	\[
	\beta_4>0.324,\qquad \beta_8>0.208.
	\]

	\subsection{Three-point lines in
		\texorpdfstring{\(\AG(n,3)\) and \(\PG(n,2)\)}{AG(n,3) and PG(n,2)}}
	
	The line systems of both \(\AG(n,3)\) and \(\PG(n,2)\) are
	\(2\text{-}(v_n,3,1)\) designs. We first give a common lower bound for
	the two cases. 
	
	\medskip
	\subsubsection{ A common lower bound} For \(S\subseteq V_n\), let \(x_i\) be the proportion of
	lines meeting \(S\) in exactly \(i\) points, and put
	\(\alpha=|S|/v_n\). The standard moment equations give
	\[
	x_0+x_1+x_2+x_3=1,\qquad
	x_1+2x_2+3x_3=3\alpha,\qquad
	x_2+3x_3=3\alpha^2+o(1).
	\]
	Hence
	\[
	x_0+x_1=1-3\alpha^2+2x_3+o(1),\qquad
	x_1+x_2=3\alpha-3\alpha^2+o(1).
	\]
	Replacing \(S\) by its complement if necessary, we may assume
	\(0\leq\alpha\leq1/2\). If \(\alpha\leq1/3\), then
	\(\frac23-o(1)\leq x_0+x_1\), while if
	\(1/3\leq\alpha\leq1/2\), then
	\(\frac23-o(1)\leq x_1+x_2\). Thus in both cases
	\(\frac13-o(1)\leq\max_i x_i\).
	
	For affine lines this gives
	\[
	\left(\frac13-o(1)\right)N_n
	\leq B(\AG(n,3),1).\]
	%\leq\left(\frac38+o(1)\right)N_n,
	%where the upper bound follows from Theorem~\ref{thm:fixed-upper}. 
	
	For the projective case, this similarly gives
	\[
	\left(\frac13-o(1)\right)N_n
	\leq B(\PG(n,2),1).\]
	%\leq\left(\frac38+o(1)\right)N_n.\]
	%The exact asymptotic constant remains open.
	
	\medskip
	\subsubsection{ The lower bound is sharp in case of PG$(n, 2)$, for projective lines.} % For projective lines the lower bound is sharp.
	Fix a hyperplane
	\(H\) of \(\PG(n,2)\), and let
	\(C\) be the complement of the hyperplane \(H\), i.e.,
	\(
	C=\PG(n,2)\setminus H
	\) is isomorphic to \(\AG(n,2)\).
	Thus \(C\)  contains no projective
	line. Indeed, every line is either contained in \(H\), in which case it
	has no point in \(C\), or meets \(H\) in exactly one point, in which
	case its other two points lie in \(C\). The proportion of lines
	contained in \(H\) is
	\[
	\frac{{n\brack2}_2}{{n+1\brack2}_2}
	=\frac14+o(1),
	\]
	so a proportion \(3/4+o(1)\) of the lines meet \(C\) in two points.
	
	Choose \(S\subseteq C\) with
	\(|S|=(2/3+o(1))|C|\). Since the lines not contained in \(H\)
	correspond to the unordered pairs of points of \(C\), the proportions
	of these lines meeting \(S\) in \(0,1,2\) points are, respectively,
	\[
	\frac19+o(1),\qquad \frac49+o(1),\qquad \frac49+o(1).
	\]
	It follows that
	\[
	x_0=\frac14+ \frac34\cdot \frac19+o(1)=\frac13+o(1),\qquad
	x_1=\frac34\cdot \frac49+o(1)=\frac13+o(1),\qquad
	x_2=\frac34\cdot\frac49+o(1)=\frac13+o(1),
	\]
	while \(x_3=0\). Consequently,
	\[
	B(\PG(n,2),1)
	=\left(\frac13+o(1)\right)N_n.
	\]
	
	\subsubsection{The lower bound \(1/3\) is not sharp in the case of AG\((n,3)\)  }
	{We retain the notation introduced in the common lower-bound argument
		above. For affine lines, the supersaturation theorem of Fox and Lovász
		\cite[Theorem~3]{FoxLovaszATR} gives the additional bound
		\[
		\alpha^{C_3}-o(1)\leq x_3,
		\]
		where \(C_3\approx13.901\) denotes the exponent appearing in their
		theorem. Let
		\[
		e=3^{-C_3}.
		\]
		Since replacing \(S\) by its complement sends \(x_i\) to \(x_{3-i}\),
		we may assume \(0\leq\alpha\leq1/2\). We claim that
		\[
		\frac13+e-24e^2-o(1)\leq\max_i x_i.
		\]
		
		If \(0\leq\alpha\leq1/3\), then the function
		\(1-3\alpha^2+2\alpha^{C_3}\) is decreasing on this interval. Hence
		\[
		\frac23+2e-o(1)
		\leq1-3\alpha^2+2\alpha^{C_3}-o(1)
		\leq x_0+x_1,
		\]
		and therefore
		\[
		\frac13+e-o(1)\leq\max_i x_i.
		\]
		
		It remains to consider \(\alpha=1/3+h\), where
		\(0\leq h\leq1/6\). The moment equations also give
		\[
		x_1=3\alpha-6\alpha^2+3x_3+o(1).
		\]
		If \(0\leq h\leq2e\), then \(\alpha^{C_3}\geq e\), and hence
		\[
		\frac13+e-24e^2-o(1)
		\leq\frac13-h-6h^2+3e-o(1)
		\leq x_1.
		\]
		Finally, if \(2e\leq h\leq1/6\), then
		\[
		\frac13+e-6e^2-o(1)
		\leq\frac13+\frac h2-\frac{3h^2}{2}+o(1)
		=\frac{x_1+x_2}{2}
		\leq\max_i x_i.
		\]
		
		Thus, with
		\[
		\nu_{\mathrm{FL}}
		:=e-24e^2
		\approx2.3308573268\cdot10^{-7},
		\]
		we obtain
		\[
		\left(\frac13+\nu_{\mathrm{FL}}-o(1)\right)N_n
		\leq B(\AG(n,3),1).
		\]
		The exact asymptotic constant remains open.}
	
	\section[Explicit constructions in affine planes of prime order] {Explicit constructions in affine planes of prime order}
	\label{sec:constructions in affine planes}
	
	In this section we present three explicit constructions, which  although not as strong as the random construction, we expect to be close to optimal.
	
	The affine plane constructions below also give constructions for the original problem in Desarguesian projective planes. Indeed, we may identify \(\AG(2,p)\) with \(\PG(2,p)\setminus\ell_\infty\). Every affine line has a unique projective completion, with the same intersection size with \(S\subseteq\AG(2,p)\), while the only additional projective line is \(\ell_\infty\), which is a \(0\)-secant of \(S\). Thus the affine and projective intersection size distributions differ by only one line, and in particular have the same asymptotic bounds. In higher dimensions this last observation no longer holds in the same form, since the hyperplane at infinity contains many subspaces.
	
	\subsection[Parabola-based constructions and character sums]{ {Parabola-based constructions and character sums}}
	
	In a finite field $\F_p$ of order $p$ where $ p$ is a prime, every element uniquely corresponds to an integer in the interval $[0, p-1]$. This naturally induces an order relation $<$ on $\F_p$.
	
	%The first was investigated by Adriaensen and the second author \cite{AW}, and we conjecture that is off from the upper bound on $\min_{S\subseteq \Pi_q} s_m(S) $ by at most a logarithmic factor $\log(q)$.
	In  our constructions below, we will assume that $p > 3$. %\todo{igazabol $p>2$, csak akkor az $\alpha = 1/4$-el picit magyarazkodni kell.} 
	
	\begin{construction}[Points under the parabola]\label{parabola}
		Let  $\alpha \in \F^*_p$, $\beta$, $\gamma \in \F_p$, and let $S$ be defined as %a set of points in $\AG(2, p)$
		\[
		S := \{(x, y) \in \F_p^2 : \alpha x^2 + \beta x + \gamma  > y \} .
		\]% \todo{rel.jel!}
	\end{construction}
	
	In \cite{AW}, Adriaensen and Weiner studied the above point set and proved the next result. They defined the projection function of $S$ from direction $(d) \in \F_p$ to be $\pr_{S,d}:\F_p \to \NN$ that maps $b \in \F_p$ to the number of points of $S$ on the line $Y=dX+b$, i.e., the secant size: $\pr_{S,d}(b)=|S\cap [Y=dX+b]|$.
	Also, let $\chi(x)$ denote the Legendre symbol.%, i.e, $\chi(0)=0$ when $x=0$, $\chi(x)=1$ if $x \neq 0$ is a square, $\chi(x)=-1$ if $x$ is a non-square.
	%\todo{ezt a magya\-rázatot leszedtem.}
	\begin{result}
		\label{res: parabola}
		Let $p>3$ prime. For the set $S$ of Construction \ref{parabola} and for any $d \in \F_p^*$ and $b \in \F_p$, 
		\begin{itemize}
			\item[(1)]{the projection function $\pr_{S,d}$ is a cyclic shift of $\pr_{S,1}$;}
			\item[(2)]{the image of $\pr_{S,d}$ is an interval whose length lies between $\frac{\sqrt p}{2 \pi}$ and $\sqrt p \ln(p)$;}
			\item[(3)]{
				$\pr_{S,d}(b+1)-\pr_{S,d}(b) = -\chi( (\beta-d)^2 + 4\alpha(b+1-\gamma) )$.
			} In particular, $|\pr_{S,d}(b+1)-\pr_{S,d}(b)|\le 1$.
		\end{itemize}
	\end{result}
	
	Note that this result implies that apart from the number of points of $S$ on the line at infinity and the vertical and horizontal lines, the possible intersection sizes on the remaining lines form an interval of length $O(\sqrt p \log p)$. This is consistent with the behavior suggested by the lower bound argument in Theorem \ref{leggyakoribb}. In order to bound the most frequent intersection number, we need to understand the frequency pattern of the lines of one parallel class. Hence, let us define prefix sums
	\[
	\Psi(a, t) = \sum_{j=0}^t \chi(a+j)
	\] for $t\in \{0,1,\ldots, p-1\}$.

	For now, let $\alpha = 1 /4 $, $\beta = \gamma = 1$; and so let $S$ be the points "under" the parabola $y=\frac {1}{4}x^2 + x + 1$ and let us take a closer look at the intersection numbers of the parallel class with slope $1$. Result \ref{res: parabola} gives $$\pr_{S,1}(b+1)-\pr_{S,1}(b) = -\chi(b).$$ Iterating this identity gives 
	\[
	\pr_{S,1}(t+1)
	=
	\pr_{S,1}(0)-\sum_{j=0}^{t}\chi(j)
	=
	\pr_{S,1}(0)-\Psi(0,t)
	\]
	
	for \(0\leq t\leq p-1\). Here \(t=p-1\) corresponds to the intercept
	\(b=0\) modulo \(p\), since \(\sum_{j=0}^{p-1}\chi(j)=0\).
	Since $\pr_{S,d}$ is a cyclic shift of $\pr_{S,1}$, in particular, $\pr_{S,d}(b)=\pr_{S,1}(b+(d-1)^2)$ in our choice for the parabola, we obtain the following corollary.
	%Specifically, the number of lines in a parallel class having a specific intersection size $k$ corresponds to the frequency of a particular value h in the sequence of prefix sums $\Psi(a,t)$. 
	\begin{corollary} \label{cor:parabola_freq}
		Let $S$ be the point set defined by the parabola $y = \frac{1}{4}x^2 + x + 1$ in $\AG(2,p)$ as in Construction \ref{parabola}. 
		For each \(k\), let \(L_k^*(S)\) denote the set of affine
		\(k\)-secants of \(S\) that are neither vertical nor horizontal. Then
		$$ |L_k^*(S)| = (p-1) \cdot |\{t \in \mathbb{F}_p : \Psi(0, t) = \pr_{S,1}(0) - k\}|, $$ 
		where $\Psi(a,t) = \sum_{j=0}^{t} \chi(a+j)$ is the prefix sum of the Legendre symbols. 
	\end{corollary}
	
	The P\'olya--Vinogradov inequality (see e.g.~\cite{MontgomeryVaughan}) provides the $\sqrt p \ln p$ upper bound on the range of this character sum, while the results of S\'{a}rk\"{o}zy \cite{Sarkozy:77} and Sokolovski\u{\i} \cite{Sokolovskii:82} provide the $\frac{\sqrt p}{2 \pi}$ lower bound; Result \ref{res: parabola} $(2)$ is a consequence of these bounds. 
	One may view the prefix sums as "random" walk on the integers, except that the increments are prescribed by the Legendre symbol. 
	The above results bound the range of the walk, expressing that character sums behave rather similarly to random walks in some aspects. However, for our purposes we need to control how often the walk attains a given value, since this frequency determines the most common secant size. This leads us to the following conjecture:
	
	\begin{conj}[Prefix sum recurrences]\label{returns}
		Let $\Psi(a,t)$ be defined via the Legendre symbol modulo a prime $p$.
		Then $$| \{t: \Psi(a,t)=0\} |= O(\sqrt{p}\log^A p),$$ independent of the value of $a\in \F_p$ for an absolute constant $A>0$.
		A stronger related question is whether
		\[
		\max_{h\in\mathbb Z}
		\bigl|\{\,0\le t\le p-1:\Psi(0,t)=h\,\}\bigr|=
		O( \sqrt p\,(\log p)^A)
		\]
		for some absolute constant $A>0$?
	\end{conj}

	If  Conjecture \ref{returns} holds, it shows that  Construction \ref{parabola} is only a logarithmic factor off from the lower bound. For further details on the statistical behavior of non-principal Dirichlet characters and their sums, we refer to \cite{char_sum}.

	The next construction was suggested by Tam\'as Sz\H onyi. Interestingly, bounding its most frequent intersection size again leads to the study of prefix-sum recurrences. 
	
	\begin{construction}[Family of parabolas]\label{parabola2}
		Let \(c\in(0,1)\) be fixed constant (independent of \(p\)), let \(p>2\) be a prime,
		and put  $a=\lfloor cp\rfloor \in \F^*_p$ and let $S$ be the following set in $\AG(2, p)$.
		\[
		S:=\{(x,x^2+t): x\in\mathbb F_p,\ 0\le t<a\}\subset \AG(2,p).
		\]
	\end{construction}
	
	For any non-vertical line $\ell: y = mx + b$, the number of its intersection points with a single parabola $y = x^2 + t$ is given by $1 + \chi(m^2 - 4(t - b))$, which corresponds to the number of solutions to the quadratic equation $x^2 - mx + (t - b) = 0$ over $\mathbb{F}_p$. Let $c\in (0,1)$ be a fixed constant (independent of $p$). Summing over \(0\leq t<a=\lfloor cp\rfloor\), the total intersection size $n_{\ell}$ is
	\begin{equation}
		n_{\ell} = \sum_{t=0}^{a-1} (1 + \chi(m^2 + 4b - 4t)) = a + \sum_{t=0}^{a-1} \chi((m^2 + 4b) - 4t).
	\end{equation}
	
	Let us introduce $$\Phi(u,a):=\sum_{t=0}^{a-1} \chi(u-t).$$
	
	For fixed slope $m$, varying the intercept $b$ shifts the initial point of the character sum in the above equation. Hence, after a suitable transformation, we get that within a given parallel class, the multiplicity of an intersection size $k$ is determined by how often the prefix sum $\Phi$ attains the value $k - a$. %So again, if Conjecture \ref{returns} holds, the most frequent intersection size of Construction \ref{parabola2} would only be a logarithm factor off from the lower bound. 
	{So again, the most frequent intersection size of Construction \ref{parabola2} depends on the distribution of character sums, this time of fixed length $a$,  $\Phi(u,a)=\sum_{t=0}^{a-1} \chi(u - t)$ for $u\in \F_p$. For  recent results on such character sums over moving intervals, see \cite{Harper1, Huss}.}

	\subsection[An elliptic-curve construction]{{An elliptic-curve construction}} %of a set where the most common intersection multiplicity size appears O$(p^{3/2 + \epsilon})$ times}
	
	To bound the frequency of the most common intersection size for the set $S$ in the next example, we first recall some key results on elliptic curves; for more details see Lenstra \cite{Lenstra}. For \(a,b\in \mathbb F_p\) satisfying \(4a^3+27b^2\neq 0\), let \(E_{a,b}\) (or simply $E$) denote the elliptic curve given by the Weierstrass equation
	
	\[
	Y^2 = X^3 + aX + b.
	\]
	
	The set of $\mathbb{F}_p$-rational points of $E$, that is $E( \mathbb{F}_p )$, is the set of points of PG$(2, p)$ satisfying the homogeneous equation:
	\[
	E(\F_p) = \left\{(x, y, z) \in \PG(2, p) : y^2z = x^3+ axz^2 + bz^3\right\}
	\]

	Note that $E(\F_p)$ always contains exactly one point with $z=0$. By Hasse's Theorem \cite{Hasse1, Hasse2}, the number of points is given by $|E(\mathbb{F}_p)| = p + 1 - t$, where $|t| \leq 2\sqrt{p}$.
	Two elliptic curves $E_{a,b}$ and $E_{a', b'}$ over $F_p$ are isomorphic if there exists $u \in \F_p^*$ such that $a' = u^4a$ and $b'=u^6b$. Following \cite{Lenstra}, we use the notation $\# '$ to denote a sum over isomorphism classes where each class $E$ is weighted by $1/|\mathrm{Aut}_{\F_p}(E)|$, where $\mathrm{Aut}_{\F_p}(E)$ denotes the group of automorphisms of $E$ defined over $\F_p$. 
	
	\begin{result}[\cite{Lenstra}, Proposition 1.9 (a)]
		Let $p > 3$ be a prime. There exists a constant $c$ such that for any set of integers $N$ where $|n - (p + 1)| \leq 2\sqrt{p}$ for all $n \in N$, we have:
		\[
		\# ' \{ \text{isomorphism classes } E: |E(\mathbb{F}_p)| \in N \} \leq c \cdot |N| \cdot \sqrt{p} \log p (\log \log p)^2.
		\]
	\end{result}
	
	Since $2 \leq |\mathrm{Aut}_{\mathbb{F}_p}(E)| \leq 6$ for $p > 3$, each weight is at least $1/6$. Furthermore, each isomorphism class corresponds to at most $(p-1)/2$ distinct pairs $(a, b)$. This leads to the following corollary.
	
	\begin{corollary}
		\label{cor:lenstra}
		Let \(p>3\) be a prime and let \(N\) be a set of integers such that $$|n-(p+1)|\le 2\sqrt p$$ for all \(n\in N\). Then the number of elliptic curves \(E_{a,b}\) for which \(|E_{a,b}(\mathbb F_p)|\in N\) is
		$
		O(|N| p^{3/2} \log p (\log \log p)^2).
		$
	\end{corollary}
	
	Finally, we can prove the following theorem.
	\begin{theorem}
		Let $p > 3$ be a prime. In $\AG(2, p)$, let $S$ be as follows:
		\[
		S := \left\{ (x, v) \in \mathbb{F}_p \times \mathbb{F}_p : \exists y \in \mathbb{F}_p \text{ such that } x^3 - v = y^2 \right\}. 
		\]
		
		Then the number of occurrences of the most frequent intersection multiplicity of $S$ with lines is $O(p^{3/2} \log p (\log \log p)^2)$.
	\end{theorem}
	
	\begin{proof}
		Let $\ell$ be a non-vertical line $v = mx + b$ with $-4m^3 + 27b^2 \neq 0$. The intersection multiplicity $n_{\ell} = |S \cap \ell|$ is the number of $x \in \mathbb{F}_p$ for which $x^3 - mx - b$ is a quadratic residue. This relates to the rational points of the elliptic curve $E: Y^2 = X^3 - mX - b$ as follows:
		\[
		|E(\mathbb{F}_p)| = 2n_{\ell} + 1 - Z_{m,b},
		\]
		where $1$ accounts for the point at infinity and $Z_{m,b} \in \{0, 1, 2, 3\}$ is the number of roots of $X^3-mX-b = 0$. So a fixed intersection multiplicity $n_{\ell}$ can only occur if the
		associated elliptic curve has a total number of $\mathbb{F}_p$-rational points: $2n+1, 2n, 2n-1, 2n-2$.  
		Therefore the result now follows by Corollary \ref{cor:lenstra}, since the $O(p)$ cases of vertical or singular lines are negligible.
	\end{proof}

	\section[Legitimate colorings of linear hypergraphs]%\zsst{Applications and related problems}\zs
	{Legitimate colorings of linear hypergraphs}\label{sec:appl}
	
	Recall that  the \textit{ legitimate coloring of a hypergraph} is a coloring of the vertex set by $k$ colors such that all hyperedges are distinguishable by their color distribution, in other words, by their multiplicity list $c_1(H), c_2(H), \ldots, c_k(H)$, where $c_i(H)$ denotes the number of vertices in hyperedge $H$ colored with the $i$-th color. Let $\chi_L(H)$ denote the minimum integer $k$ such that there exists a legitimate
	coloring of a hypergraph $H$ by $k$  colors.
	Alon and Füredi proved that $5\le \chi_L(\Pi_q)\le 8$, provided that $q$ is sufficiently large \cite{Alon-Furedi}. 
	
	A naive lower bound on  the legitimate coloring number would follow from the following argument. Let us consider the line set $\cL$ of the projective plane $\Pi_q$. Then by Theorem \ref{leggyakoribb}, we find at least $\min_{ S\subseteq \Pi_q} \max_k \{|L_k(S)|\}$ lines which are identical in the first coordinate of the color multiplicity list.  In order to distinguish all the lines, we will need more colors, thus let us consider the largest non-distinguished line set $\cL'\subset \cL$ and taking into consideration the next color. We should keep introducing new colors at least up to the point when the largest undistinguished line set by the previously applied colors has more than one line.

	The next {proposition} would yield the last step of this procedure and  can be viewed as the legitimate version of the celebrated Erdős-Faber-Lovász conjecture. Recall that it states the following.
	
	\begin{conj}[Erdős-Faber-Lovász, \cite{erdos}]
		If a linear hypergraph $\HH=(V,   \mathcal{E})$ has $n$ edges, each of size at most $n$, then $\HH$ can be colored properly by $n$ colors.    
	\end{conj}
	
	This famous conjecture was fairly recently resolved in the case when $n$ is sufficiently large \cite{Kang}.
	
	\begin{proposition}
		Let $\mathcal{H}$ denote an arbitrarily chosen $n$-uniform linear hypergraph with $n$ edges.  Then the legitimate coloring number of  $\mathcal{H}$ is $2$.
	\end{proposition} 
	
	\begin{proof} Let the edges of $\mathcal H$ be ordered arbitrarily as
		$F_1,F_2,\dots,F_n$. We provide a suitable blue-red coloring of the vertex set in two phases.
		In the first phase we first construct a preliminary red-blue coloring. Color all vertices of
		$F_1$ blue. Suppose that the vertices of  $F_1,\dots,F_{i-1}$ have already been colored.
		If $i$ is odd, color every previously uncolored vertex of $F_i$ blue. If
		$i$ is even, color every previously uncolored vertex of $F_i$ red.
		
		We claim that after this first phase, every odd-indexed edge has blue
		majority and every even-indexed edge has red majority, in the following
		stronger sense:
		\[
		\ensuremath{n-\lfloor i/2\rfloor\le |F_i\cap B|}
		\quad\text{if } i \text{ is odd},
		\]
		and
		\[
		|F_i\cap B|\le i/2
		\quad\text{if } i \text{ is even},
		\] where $B$ indicates the set of blue vertices.
		Indeed, if $i$ is odd, then the only red vertices of $F_i$ can come from
		intersections with previous even-indexed edges. There are only
		$\lfloor i/2\rfloor$ such edges, and by linearity each contributes at most
		one vertex to $F_i$. Hence $F_i$ contains at most $\lfloor i/2\rfloor$ red
		vertices, and so at least $n-\lfloor i/2\rfloor$ blue vertices. The even
		case is analogous: if $i$ is even, then the only blue vertices of $F_i$
		can come from intersections with previous odd-indexed edges, of which
		there are $i/2$, so $F_i$ contains at most $i/2$ blue vertices.

		%we greedily color the vertices of the hypergraph blue and red as follows. Let us color all the vertices of $F_1$ blue.  Then, let us color the vertices of $F_i$ $(i>1)$ after  all the vertices of $\cup_{j<i} F_j$ are colored, such that the number of blue vertices in $F_i$ is $i/2$ or less if $i$ is even and $n-\lfloor i/2\rfloor $ or more if $i$ is odd, obtaining edges of blue and red majority alternately. This can be obtained since $F_i$  intersects at most $\lfloor i/2\rfloor$ hyperedges of blue majority from $\{F_j : j<i\}$ and $\lfloor (i-1)/2\rfloor$ hyperedges of red majority, we can precolor all the remaining vertices of $F_i$ blue if $i$ is odd and red if $i$ is even to fulfill the required condition. Indeed, such a coloring guarantees that a red vertex  appears on a hyperedge $F_i$ with blue majority only at most $\lfloor (i-1)/2\rfloor$ times. Similarly,  a {blue} vertex  appears on a hyperedge $F_i$ with {red} majority only at most $\lfloor i/2\rfloor$ times.\\
		In the second phase we wish to recolor a subset of the vertices to obtain the required property. For an edge
		$F_i$, we call a vertex $v\in F_i$ \textit{private} for $F_i$ if $v$ belongs to no
		edge of $\mathcal H$ other than $F_i$.
		We now modify the coloring so that
		\[
		|F_i\cap B|=
		\begin{cases}
			n-\lfloor i/2\rfloor, & i \text{ odd},\\
			i/2, & i \text{ even}
		\end{cases}
		\]
		holds for the number of blue vertices of the edges. 
		We do this successively for $i=1,2,\dots,n$. We show that this can be done via recoloring only private vertices in
		$F_i$.
		
		% so that $F_i$ contains exactly $i/2$ {blue} vertices if $i$ is even and $n-\lfloor i/2\rfloor$ if $i$ is odd. We guarantee this set of conditions consecutively on $F_1, F_2, \ldots, F_n$ by recoloring a subset of $F_i\setminus\cup_{j< i}F_j$ at step $i$.\\
		
		For a pair of intersecting edges $F_i, F_j$, let us call $F_j$ \textit{captured w.r.t. $F_i$}, if $\exists k<j$, such that $(F_i\cap F_j)\in F_k$. In other words,   if $F_i\cap F_j=\{v\}$, then $F_j$ is captured with respect
		to $F_i$ when the intersection vertex $v$ already appeared in an earlier
		edge. %It is easy to see that the coloring of the first phase yields  $$t_i=|\{j<i \colon F_j \text{ \ is \ captured \ w.r.t. \ } F_i\}|.$$ For a vertex $v$, 
		
		%For $v\in F_i$, define\[m_i(v):=|\{j: v\in F_j\}|.\]Thus $m_i(v)=1$ if and only if $v$ is private for $F_i$.
		
		Let $R_i$ denote the set of private vertices of $F_i$, let
		$$
		C_i:=\{j<i: F_j \text{ is captured with respect to } F_i\},$$
		and let us consider the disjoint edges w.r.t. $F_i$,
		$$
		D_i:=\{j:F_i\cap F_j=\emptyset\}.$$
		We claim that
		$
		\ensuremath{|C_i|+|D_i|+1\le |R_i|}
		$ holds. Indeed, since $\mathcal H$ is linear, every edge $F_j\neq F_i$ is either
		disjoint from $F_i$ or meets $F_i$ in exactly one vertex.  Every
		non-private vertex of $F_i$ accounts for at least one such intersecting
		edge, while every member of $C_i$ accounts for an additional one.
		%every edge $F_j\neq F_i$ is either disjoint from $F_i$ or meets $F_i$ and contributes exactly one incidence with a vertex of $F_i$. Therefore
		\begin{equation}
			n-1=|\{j\in[1,n] : j\neq i\}|\ge |D_i|+(n-|R_i|)+|\{j<i: F_j \text{ is captured with respect to } F_i\}|.   
		\end{equation}
		As $C_i \subseteq \{j: F_j \text{ is captured with respect to } F_i\}$, the inequality in turn follows.

		Suppose first that $i$ is odd.  After the first phase, $F_i$ has at least
		$n-\lfloor i/2\rfloor$ blue vertices. If it has $n-\lfloor i/2\rfloor+t_i$ blue vertices \((0\le t_i)\), then it has $\lfloor i/2\rfloor-t_i$
		red vertices. The red vertices of $F_i$ must lie in previous even-indexed
		edges, and the reason for a previous even edge $F_j$ failing to contribute a red
		vertex to $F_i$ is that either it is disjoint from $F_i$, or its
		intersection with $F_i$ has already been captured by an earlier edge. Hence the number of additional red vertices needed in $F_i$ is at most $|D_i|+|C_i|.$ Since $i$ is odd, all private vertices of $F_i$ were
		colored blue in the first phase. We may therefore recolor the required
		number of them from blue to red, as $
		\ensuremath{|C_i|+|D_i|+1\le |R_i|}
		$ holds.
		
		% $\sum_{v\in V} m_i(v)\le |F_i|+| \{j\neq i : j\in [1,n]\}|\le 2n-1$ 

		The case where $i$ is even is similar.
		Finally, the prescribed values \[ |F_i\cap B|= \begin{cases} n-\lfloor i/2\rfloor, & i \text{ odd},\\ i/2, & i \text{ even}, \end{cases} \] are precisely the integers \(1,2,\ldots,n\), in some order. Hence the edges have pairwise distinct color multiplicity lists, and the resulting 2-coloring is legitimate.
		%  In order to do so, observe that if an edge $F_i$ of majority blue contains more that  $n-\lfloor i/2\rfloor $ blue vertices by $t$, that means $F_i$ actually intersected the red-majority edges of index  less than $i$ in $t$ less vertices than the expected $\lfloor (i-1)/2\rfloor$. Hence, $F_i$ contains at least $t$ points of degree $1$ in $\cup_{j\le i}F_j$, due to the uniformity and the linearity of $\mathcal{H}$. The alternation of the color of $t$ of these $t$ points enables to set the number of blue colors on $F_i$ to the required value, for every odd value of $i$. The recoloring for the case of hyperedges of red majority, i.e., when $i$ is even, is similar. 
	\end{proof}

	\newpage
	
	\begin{appendices}
		\section[An elementary lower-bound argument for projective planes]{{An elementary lower-bound argument for projective planes}}
		{For completeness, we give an elementary argument that yields the same asymptotic lower-bound constant as Theorem~\ref{leggyakoribb}, with a slightly weaker lower-order term.}

		{We start by proving a variance formula for the secant sizes of an arbitrary point set.} This essentially originates from Bruen \cite{Br}.
		\begin{lemma}
			\label{lem:standard}
			Let $\Pi_q$ be a projective plane of order $q$  and let $S\subset \Pi_q$ be a set of points.
			{For each line $\ell\in \cL$, let $n_{\ell}$ denote the number of points of $S$ on $\ell$ and let} $\mu$ be the average number of points on a line, i.e., 
			$\mu: =  \frac{1}{q^2+q+1}\sum_{\ell\in \cL} n_{\ell}$.
			
			Then
			\[
			\sum_{\ell\in \cL}(n_{\ell}-\mu)^2 
			= q|S|\left(1-\frac{|S|}{q^2+q+1}\right).%q(|S| -\frac{|S|^2}{q^2+q+1}).
			%= |B|(|B|+n)-%\frac{|B|^2(n+1)^2}{N}.
			\]
		\end{lemma}
		
		\begin{proof}  %{\zsst{Let $N$ denote the total number of lines of $\Pi_q$, i.e. $N= q^2+q+1$.}} 
			Let $N:=q^2+q+1$ denote the total number of lines.
			Let us recall the standard equations:
			\begin{equation}
				\sum_{\ell\in \cL} n_{\ell} = |S|(q+1), \ \
				\sum_{\ell\in \cL} n_{\ell}(n_{\ell}-1) = |S|(|S|-1).
			\end{equation}
			
			It follows that 
			\[
			\mu=\frac{\sum_{\ell\in \cL}n_{\ell}}{N}=\frac{|S|(q+1)}{N}
			\]
			and 
			\[
			\sum_{\ell\in \cL} n_{\ell}^2
			= |S|(|S|-1)+|S|(q+1)=|S|(|S|+q),
			\]
			hence
			\begin{equation}
				\sum_{\ell\in \cL}(n_{\ell}-\mu)^2
				=\sum_{\ell\in \cL}n_{\ell}^2-2\mu\sum_{\ell\in \cL}n_{\ell}+N\mu^2
				=\sum_{\ell\in \cL}n_{\ell}^2-N\mu^2
				=|S|(|S|+q)-\frac{|S|^2\,(q+1)^2}{N}.
			\end{equation}
			{Since \((q+1)^2=N+q\), the last expression equals
				\(q|S|(1-|S|/N)\), as required.}
			
		\end{proof}

		\begin{proposition}\label{prop:general_lower_bound}
			Let $\Pi_q$ be a projective plane of order $q$,   $N:=q^2+q+1$ denote the total number of lines,  and let $S$ be a set of points in $\Pi_q$. Then there exists an integer $k$ such that
			
			\[
			\frac{N^{3/2}}{\sqrt{12q|S|(1-\frac{|S|}{N}) + { {13N}} }}
			\le |L_k(S)|.
			\]
		\end{proposition}
		\begin{proof}
			Let $M$ denote the maximum cardinality of the sets of $k$-secants,
			\[
			M:=\max_{k\in\{0,\dots,q+1\}} |L_k(S)|.
			\]   Let us bound from below the variance \[V:=
			\sum_{\ell\in \cL}(n_{\ell}-\mu)^2
			=q|S|\left(1-\frac{|S|}{N}\right),
			\] where $\mu=\frac{\sum_{\ell\in \cL}n_{\ell}}{N}=\frac{|S|(q+1)}{N}$ is the average intersection size.
			Write $\mu=\lfloor\mu\rfloor+\alpha$ with $\alpha\in[0,1)$.
			For each integer \(j\ge 0\), consider the pairs 
			$$U_j:=\{
			\lfloor\mu\rfloor-j,\lfloor\mu\rfloor+1+j
			\}.$$
			Their squared distances from $\mu$ satisfy
			\[
			\ensuremath{2\left(j+\frac12\right)^2
				\le 2\left(j+\frac12\right)^2+2\left(\alpha-\frac12\right)^2
				=(\alpha+j)^2+(1-\alpha+j)^2}.
			\]
			Define $t$ and $r$ by the Euclidean division
			\[
			N=2Mt+r,\qquad 0\le r<2M.
			\]
			Then the variance is minimized when $|L_k|=M$ for $k\in \bigcup_{j=0}^{t-1} U_j$
			and $\sum_{k\in U_{t}} |L_k|=r$. Observe that $(\alpha+t)^2$ and  $(1-\alpha+t)^2$ are both at least $t^2$.
			Hence
			
			\[
			2M\sum_{j=0}^{t-1}\left(j+\frac12\right)^2
			+
			rt^2 \le V.
			\]
			
			Using
			\[
			2\sum_{j=0}^{t-1}\left(j+\frac12\right)^2
			= \frac{1}{2}\sum_{j=0}^{t-1}(2j+1)^2= 
			\frac{t(4t^2-1)}{6},
			\]
			
			we get
			\[
			\frac{Mt(4t^2-1)}{6}
			+
			rt^2 \le V.
			\]
			Recalling that \(r<2M\), a straightforward calculation shows that
			\[
			\frac{Mt(4t^2-1)}{6}
			+
			rt^2
			-
			\left(\frac{N^3}{12M^2}-\frac{13N}{12}\right)
			=
			\ensuremath{\frac{ 6Mt(4M^2-r^2)  + r(13M^2- r^2)}{12M^2} \ge 0}.
			\]
			Therefore
			\[
			\frac{N^3}{12M^2}-\frac{13N}{12} \le V,
			\]
			which implies
			\[
			\frac{N^{3/2}}{\sqrt{12V+13N}} \le M.
			\]
			Substituting $
			V=q|S|\left(1-\frac{|S|}{N}\right)$
			gives the required bound.\end{proof}
		
		\begin{corollary}
			\label{cor:gen lower bound}
			Let $\Pi_q$ be a projective plane of order $q$ and let $S$ be a set of points in $\Pi_q$. Then there exists an integer $k$ such that
			
			\[
			\frac{q^2+q+1}{\sqrt{3q+13}}
			\le |L_k(S)|
			\] 
		\end{corollary}
		
		\begin{proof} The proof follows from Proposition \ref{prop:general_lower_bound} after applying the AM-GM inequality  
			$$ N\frac{|S|}{N}(1-\frac{|S|}{N})\le N\frac{1}{4},$$%, where $N=q^2+q+1$,} 
			
			which in turn gives
			\[
			\frac{N^{3/2}}{\sqrt{3qN+13N}} \le |L_k(S)|
			.
			\] 
		\end{proof}
		{This gives the same leading lower-bound constant as Theorem~\ref{leggyakoribb}.}
	\end{appendices}

\end{document}